\documentclass[a4paper, 12pt]{amsart}
\usepackage{amsmath}
\usepackage{amsthm}
\usepackage{amssymb}
\usepackage{xypic}
\usepackage[only,mapsfrom]{stmaryrd}
\usepackage{graphicx}

\newtheorem{theorem}{Theorem}[section]
\newtheorem{lemma}[theorem]{Lemma}

\newtheorem{conjecture}[theorem]{Conjecture}

\theoremstyle{definition}

\theoremstyle{remark}

\numberwithin{equation}{section}

\DeclareMathOperator{\rank}{rank}

\title{On a conjecture of Stolz in the toric case}

\author{Michael Wiemeler}
\address{
Mathematisches Institut\\ Universit\"at M\"unster\\Einsteinstrasse 62\\D-48149 M\"unster\\Germany
}
\email{wiemelerm@uni-muenster.de}
 \thanks{The research for this paper was funded by the Deutsche Forschungsgemeinschaft (DFG, German Research Foundation) under Germany's Excellence Strategy EXC 2044 –390685587, Mathematics M\"unster: Dynamics–Geometry–Structure and through CRC1442 Geometry: Deformations and Rigidity at University of M\"unster.}



\date{\today}

\begin{document}
\begin{abstract}
  In 1996 Stolz conjectured that a string manifold with positive Ricci curvature has vanishing Witten genus.
  Here we prove this conjecture for toric string Fano manifolds and for string torus manifolds admitting invariant metrics of non-negative sectional curvature.
\end{abstract}

\maketitle

\section{Introduction}

For a simply connected closed manifold \(M\) with positive scalar curvature there is no known obstruction to \(M\) admitting a metric of positive Ricci curvature.
However from a constructive point of view it is much easier to construct metrics of positive scalar curvature than metrics of positive Ricci curvature.

Given this situation Stolz conjectured in \cite{MR1380455} that the Witten genus of a closed string manifold admitting a metric of positive Ricci curvature vanishes.
He motivated this conjecture by an heuristic argument involving the free loop space of the manifold.
His conjecture has been verified in several special cases, e.g. for string complete intersections in complex projective space by Landweber and Stong \cite[p. 87]{MR1189136}  or more generally in certain hermitian symmetric spaces by F\"orster \cite{foerster-phd}, for string homogeneous spaces and for string cohomogeneity-one manifolds by Dessai \cite{MR2581907}.
Note here that in the case of string homogeneous spaces the conjecture has independently been proven by H\"ohn (unpublished).

A string manifold \(M\) is a spin manifold with \(\frac{p_1}{2}(M)=0\).
A spin manifold is an oriented manifold with vanishing second Stiefel-Whitney class.
The Witten genus is a ring homomorphism \(\Omega^{SO}\rightarrow \mathbb{Q}[[q]]\) where \(\Omega^{SO}\) denotes the oriented bordism ring.

Here we consider torus manifolds and verify the conjecture in special cases.
A torus manifold is a closed connected orientable manifold of real dimension \(2n\) with an effective action of an \(n\)-dimensional compact torus \(T\) such that there are \(T\)-fixed points in \(M\).
Moreover, we call a complex manifold toric if it is biholomorphic to a smooth complex toric variety.
For an overview about the properties of these manifolds and varieties see \cite{MR3363157} and \cite{MR2810322}.

By Yau's solution \cite{MR480350}  of the Calabi conjecture a closed connected Kähler manifold with positive first Chern class admits a Kähler metric of positive Ricci curvature. Such Kähler manifolds are called Fano manifolds.

For the case of toric Fano manifolds we prove the following theorem which implies the Stolz conjecture in this case.

\begin{theorem}
  \label{sec:introduction}
  Let \(M\) be a toric string Fano manifold of complex dimension \(n\).
  Then \(M\) is biholomorphic to \(\prod_{i=1}^n \mathbb{C} P^1\).
\end{theorem}

We also prove the following:

\begin{theorem}
  \label{sec:introduction-1}
  Let \(M\) be a string torus manifold with an invariant metric of non-negative sectional curvature. Then the Witten genus of \(M\) vanishes.
\end{theorem}

To give some background on this theorem we note that by a result of Böhm--Wilking \cite{MR2346271} the Ricci flow on a manifold with finite fundamental group evolves a metric of non-negative sectional curvature to a metric of positive Ricci curvature.
Moreover, by their classification given in \cite{MR3355120} torus manifolds with invariant metrics of non-negative sectional curvature have finite fundamental groups.
Hence they have invariant metrics of positive Ricci curvature.

These theorems answer questions raised by Dessai \cite[Section 3]{MR2581907} in the toric case.
Dessai asked to check the Stolz conjecture for string Fano manifolds and for string manifolds admitting metrics of positive or non-negative sectional curvature.

Our proof of Theorem~\ref{sec:introduction} combines a lower bound on the second Betti-number \(b_2(M)\)  of a string quasitoric manifold \(M\) from \cite{MR3080634} with the Mukai conjecture which has been proved in the toric case by Casagrande \cite{MR2228683}.

Our proof of Theorem~\ref{sec:introduction-1} combines the classification of torus manifolds with invariant metrics of non-negative sectional curvature given in \cite{MR3355120}, a classification result for certain string quasitoric manifolds due to Shen \cite{shen22} and a vanishing result from \cite{MR3599868}.

Quasitoric manifolds form a special class of torus manifolds which includes non-singular projective toric varieties.
This class of manifolds has been introduced in \cite{zbMATH04212906}.
We note that by a conjecture of Bazaikin--Matvienko \cite{MR2364619}  every quasitoric manifold should have an invariant metric of positive Ricci curvature.
Combining this with the Stolz conjecture would imply that the Witten genus of any string quasitoric manifold vanishes.
This would also be implied by a conjecture of Liu \cite{MR1331972} who asked whether the Witten genus of any string manifold with an effective \(S^1\)-action vanishes.

This paper has two more sections.
In Section~\ref{sec:proof-theor-refs-2} we prove Theorem~\ref{sec:introduction}.
Then in Section~\ref{sec:proof-theor-refs-3} we prove Theorem~\ref{sec:introduction-1}.

\section{The proof of Theorem \ref{sec:introduction}}
\label{sec:proof-theor-refs-2}

In this section we prove Theorem~\ref{sec:introduction}. We start with some preparations.

A quasitoric manifold \(M\) is a locally standard torus manifold such that \(M/T\) is face-preserving homeomorphic to a simple convex polytope.
Here a \(2n\)-dimensional torus manifold is called locally standard if the torus action on it is locally modeled on the standard action of \(T\) on \(\mathbb{C}^n\).
In this case \(M/T\) is a nice manifold with corners.

Note that every toric Fano manifold is a non-singular projective toric variety and hence a quasitoric manifold.

We first have to recall a result on string quasitoric manifolds from \cite{MR3080634}.

\begin{lemma}
\label{sec:proof-theor-refs}
  Let \(M\) be a string quasitoric manifold of real dimension \(2n\). Then
  \[b_2(M)\geq n.\]
\end{lemma}
\begin{proof}
  This follows directly from Lemma 2.1 and the proof of Proposition 2.2 in \cite{MR3080634} which was based on results in \cite{zbMATH04011454}.
\end{proof}

Next we have to recall some notations for Fano manifolds.
Let \(M\) be a complex manifold of complex dimension \(n\).

  Let \(\mathrm{Pic}(M)\) be the group of holomorphic line bundles over \(M\) modulo isomorphism.
  Then the first Chern class induces a homomorphism
  \[c_1:\mathrm{Pic}(M)\rightarrow H^2(M;\mathbb{Z}),\]
  see \cite[p. 139]{zbMATH00790015}.
  
  The Picard number \(\rho_M\) of \(M\) is defined as the rank of the image of the map \(c_1:\mathrm{Pic}(M)\rightarrow H^2(M;\mathbb{Z})\).
  
  It follows from the Kodaira vanishing theorem \cite[p. 159-160]{zbMATH00790015} that \(c_1:\mathrm{Pic}(M)\rightarrow H^2(M;\mathbb{Z})\) is an isomorphism if \(M\) is Fano.
  Hence, in this case the Picard number of \(M\) equals the second Betti number of \(M\), i.e.
  \begin{equation}
    \label{eq:2}
    \rho_M=b_2(M).
  \end{equation}

If \(M\) is a complex algebraic Fano variety one defines the pseudo-index of \(M\) as
\[\iota_M=\min\{K_M^*\cdot C;\; C\subset M \text{ rational curve}\}\in \mathbb{Z}_{>0}.\]

Here \(K_M^*=\wedge_{\mathbb{C}}^nTM\) is the anti-canonical line bundle of \(M\). Moreover, \(K_M^*\cdot C\in \mathbb{Z}\) denotes the intersection product of the divisor determined by a meromorphic section of \(K_M^*\) and \(C\).
Note that, by \cite[p. 141]{zbMATH00790015}, \[K_M^*\cdot C=\int_C c_1(K_M^*)\] is independent of the choice of the meromorphic section of \(K_M^*\).
Note, moreover, that 
\(K_M^*\cdot C > 0\) for every rational curve \(C\subset M\) since \(c_1(K^*_M)=c_1(M)\) is positive by the Fano condition.
  
  After recalling these basic facts about Fano manifolds we can state a version of the Mukai conjecture for smooth Fano varieties (for other versions of this conjecture see \cite{mukai88} and \cite{zbMATH02021087}).

  \begin{conjecture}[Mukai conjecture]
    Let \(M\) be a smooth Fano variety of complex dimension \(n\). Then we have
    \[\rho_M(\iota_M-1)\leq n\] with equality if and only if \(M\) is biholomorphic to \(\left(\mathbb{C} P^{\iota_M-1}\right)^{\rho_M}\).
  \end{conjecture}

  For toric Fano manifolds this conjecture has been shown by Casagrande \cite{MR2228683}, while it is still open in general.

  We also need the following simple lemma.

\begin{lemma}
\label{sec:proof-theor-refs-1}
  Let \(M\) be a smooth spin Fano variety. Then \(\iota_M\geq 2\).
\end{lemma}
\begin{proof}  
  By the spin condition, we have \[c_1(K^*_M)=c_1(M)\equiv w_2(M)=0\mod 2.\]
  Here \(w_2(M)\) denotes the second Stiefel--Whitney class of \(M\).
  Hence, there is \(L\in \mathrm{Pic}(M)\) such that \(K^*_M=L\otimes L\).
  Now for every rational curve \(C\subset M\) we have
  \[K_M^*\cdot C=2(L\cdot C)\in 2\mathbb{Z}.\]
  Hence the claim follows.
\end{proof}

Now we have all ingredients for the proof of  Theorem~\ref{sec:introduction}. Let \(M\) be a toric string Fano manifold of complex dimension \(n>0\). Then by a combination of Lemmas \ref{sec:proof-theor-refs} and \ref{sec:proof-theor-refs-1}, Equation~(\ref{eq:2}) and the Mukai conjecture we have
\[n\leq n(\iota_M-1)\leq b_2(M)(\iota_M-1)=\rho_M(\iota_M-1)\leq n.\]
Therefore we must have equality in all of the above inequalities.
Hence, the claim follows from the rigidity statement in the Mukai conjecture.

\section{The proof of Theorem~\ref{sec:introduction-1}}
\label{sec:proof-theor-refs-3}

In this section we prove Theorem~\ref{sec:introduction-1}.

Torus manifolds with invariant metrics of non-negative sectional curvature have been classified up to equivariant diffeomorphism in \cite{MR3355120}.
They all have finite fundamental groups by Lemma 7.1 in \cite{MR3355120}.
Since the Witten genus is multiplicative in finite coverings it suffices to prove the theorem in the case that \(M\) is a simply connected string torus manifold of positive dimension that admits an invariant metric of non-negative sectional curvature.
Then, by Theorem 1.2 of \cite{MR3355120}, \(M\) is equivariantly diffeomorphic to a quotient of a free linear action of an \(r\)-dimensional torus on a product of spheres
\[Z=\prod_{i=1}^r S^{2n_i+1}\times \prod_{i=1}^s S^{2m_i}\]
with \(\dim M=2n=2\sum_{i=1}^r n_i +2\sum_{i=1}^s m_i\), \(n_i>0\) and \(m_i>1\) for all \(i\).

It follows from the exact homotopy sequence for the fibration
\[T^r\rightarrow Z\rightarrow M\]
and the Hurewicz theorem that
\begin{equation}
  \label{eq:1}
  b_2(M)=\rank \pi_2(M)=r\leq n
\end{equation}
with equality if and only if \(s=0\) and \(n_i=1\) for all \(i\).

In the equality case \(M\) is a quasitoric manifold with orbit space diffeomorphic to the cube \([-1,1]^n\).
In this case it follows from \cite[Theorem 4.2]{shen22} that \(M\) is weakly equivariantly homeomorphic to a Bott manifold. Using \cite[Theorem 1.1]{MR4450678} one can upgrade this homeomorphism to a weakly equivariant diffeomorphism.
For \(n=1\), \(\mathbb{C} P^1\) is the only Bott manifold of real dimension \(2n\).
For \(n>1\), a \(2n\)-dimensional Bott manifold is the total space of a \(\mathbb{C} P^1\)-bundle with structure group \(S^1\) over a \((2n-2)\)-dimensional Bott manifold.
Since \(S^2=\mathbb{C}P^1\) equivariantly bounds an orientable \(S^1\)-manifold, it follows that every Bott manifold bounds an orientable manifold.
Hence, the claim follows in this case.

If we have strict inequality in (\ref{eq:1}) it follows from \cite[Theorem 3.2]{MR3599868} that the Witten genus of \(M\) vanishes.


\bibliography{refs}{}
\bibliographystyle{alpha}
\end{document}